\documentclass[12pt, leqno]{amsart}

\usepackage[OT2,T1]{fontenc}
\DeclareSymbolFont{cyrletters}{OT2}{wncyr}{m}{n}
\DeclareMathSymbol{\Sha}{\mathalpha}{cyrletters}{"58}

\usepackage{indentfirst}
\usepackage{amstext}
\usepackage{amsopn}
\usepackage{amsfonts}
\usepackage{amsmath}
\usepackage{latexsym}
\usepackage{amscd}
\usepackage{amssymb}
\usepackage{amsmath}
\usepackage[all,cmtip]{xy}
\usepackage{leftidx}
\usepackage{graphicx}
\usepackage{tikz}

=\oddsidemargin \font\teneufm=eufm10 \font\seveneufm=eufm7
\font\fiveeufm=eufm5
\newfam\eufmfam
\textfont\eufmfam=\teneufm \scriptfont\eufmfam=\seveneufm
\scriptscriptfont\eufmfam=\fiveeufm

\def\GG{\mathbb{G}}

\def\1{\mbox{\bf 1}}

\def\type{\mathrm{\bf type}}

\DeclareMathOperator{\Aut}{Aut}
\DeclareMathOperator{\Autext}{Autext}
\DeclareMathOperator{\Int}{Int}
 
\DeclareMathOperator{\Isom}{Isom}
\DeclareMathOperator{\Isomext}{Isomext}
\DeclareMathOperator{\Isomint}{Isomint}

 \DeclareMathOperator{\Id}{Id}

\DeclareMathOperator{\SL}{\rm SL}

\newtheorem{theorem}{Theorem}[subsection]%Theorems et al in italic font.

\newtheorem{claim}[theorem]{Claim}

\newtheorem{corollary}[theorem]{Corollary}

\newtheorem{lemma}[theorem]{Lemma}

\newtheorem{stheorem}{Theorem}[section]%Theorems et al in italic font.
\newtheorem{sclaim}[stheorem]{Claim}

\newtheorem{slemma}[stheorem]{Lemma}
\newtheorem{sproposition}[stheorem]{Proposition}
\newtheorem{sremark}[stheorem]{Remark}
\newtheorem{sremarks}[stheorem]{Remarks}

\newtheorem{fsremark}[theorem]{Remark}

\theoremstyle{definition}%Theorems et al in roman font.

\newtheorem{definition}[theorem]{Definition}

\numberwithin{equation}{section}

\def\ZZ{\mathbb{Z}}

\def\ad{\text{\rm ad}}

\def\ol{\overline}

\def\2int{\mathop{2\int}\nolimits}

\def\Spec{\mathop{\rm Spec}\nolimits}

\def\Gal{\mathop{\rm Gal}\nolimits}

\def\Int{\mathop{\rm Int}\nolimits}

\def\Aut{\text{\rm{Aut}}}

\def\Int{\mathop{\rm Int}\nolimits}

\def\Isom{\mathop{\rm Isom}\nolimits}
\def\Isomext{\mathop{\rm Isomext}\nolimits}

\def\resp.{\mathop{\rm resp.}\nolimits}

\def\limind{\mathop{\oalign{lim\cr
\hidewidth$\longrightarrow$\hidewidth\cr}}}

\def\lgr{\longrightarrow}

\font\math=cmmi10
\def\varpi{\hbox{\math\char'44}}

\def\simlgr{\buildrel\sim\over\lgr}

\def\pa{\S\kern.15em }

\def\un{\uppercase\expandafter{\romannumeral 1}}
\def\deux{\uppercase\expandafter{\romannumeral 2}}
\def\trois{\uppercase\expandafter{\romannumeral 3}}
\def\quatre{\uppercase\expandafter{\romannumeral 4}}
\def\cinq{\uppercase\expandafter{\romannumeral 5}}
\def\six{\uppercase\expandafter{\romannumeral 6}}

\def\hfl#1#2#3{\smash{\mathop{\hbox to#3{\rightarrowfill}}\limits
^{\scriptstyle#1}_{\scriptstyle#2}}}
\def\gfl#1#2#3{\smash{\mathop{\hbox to#3{\leftarrowfill}}\limits
^{\scriptstyle#1}_{\scriptstyle#2}}}

\begin{document}

\title[Maximal tori]{Examples of algebraic groups of type $G_2$ having same maximal tori}

\author{C. Beli}
\address{
   Institute of Mathematics Simion Stoilow of the Romanian Academy,
  Calea Grivitei 21,
 RO-010702 Bucharest, Romania.
}
\email{Beli.Constantin@imar.ro}

\author{P. Gille}
\address{
   Institute of Mathematics Simion Stoilow of the Romanian Academy,
  Calea Grivitei 21,
 RO-010702 Bucharest, Romania.
}
\thanks{C. Beli and P. Gille were supported by the Romanian IDEI project PCE$_{-}$2012-4-364 of the Ministry of National Education
CNCS-UEFISCIDI}
\email{pgille@imar.ro}

\author{T.-Y. Lee}
\address{Ecole Polytechnique F\'ed\'erale de Lausanne,
EPFL, Station 8,
CH-1015 Lausanne, Switzerland.
}
\email{ting-yu.lee@epfl.ch}
\date{\today}

\begin{abstract} Answering a question of A. Rapinchuk, we construct
examples of non-isomorphic  semisimple algebraic groups  $H_1$  and $H_2$ of type $G_2$
having coherently equivalent systems of maximal $k$-tori.

\medskip


\medskip

\end{abstract}

\maketitle
%\tableofcontents

\bigskip

\hfill\hfill{\it To V.P. Platonov on  his 75th birthday}
\bigskip

\section{Introduction} \label{sec_intro}

Let $k$ be a field. We say that two semisimple algebraic groups $H_1$  and $H_2$ have
same maximal tori  if
each time there is an embedding
$\iota_1: T \to H_1$ of a maximal torus $T$, then there is an embedding
$\iota_2: T \to H_2$ and conversely. This defines an equivalence class
on isomorphism classes of semisimple algebraic $k$--groups.

There are variants of this equivalence relation.
For example, one may  require  additionnally that $H_1$ and $H_2$ are isomorphic over $\ol k$,
and we say then that $H_1$ and $H_2$ are of the same genus.
The genus of  inner groups of type $A$ over arithmetic fields
has been investigated by Chernousov-Rapinchuk-Rapinchuk \cite{CRR}.

One variant  is coarser, it is the same up to isogeny for $T$ and has been
studied by  Garibaldi-Rapinchuk \cite{GR}. We say then that $H_1$ and
$H_2$ have same tori up to isogeny.

One is finer, it is due to  Prasad-Rapinchuk \cite[def. 9.4]{PR}
and roughly speaking takes into account the Galois action on the root systems
$\Phi(H_1, \iota_1(T))$ and $\Phi(H_2, \iota_2(T))$, see below (\S \ref{sect_coherent}).
We say then that $H_1$ and $H_2$ have {\it coherently equivalent systems} of maximal $k$-tori;
it  has been investigated  over number fields ({\it loc. cit.}).

Garibaldi and Saltman constructed a field $F$ and  non-isomorphic  semisimple simply connected $F$--groups $H_1,H_2$  of type
$A_1$ such that  $H_1$ and $H_2$ have  coherently equivalent systems of maximal $F$-tori. It is written in terms of quadratic
subfields of quaternion algebras and translate easily in terms of maximal tori of relevant semisimple groups, see
Lemma \ref{lem_coh_A1}.

Inspired by this construction, we  construct  examples of  non-isomorphic  semisimple $F$--groups $H_1,H_2$  of type
$G_2$ such that  $H_1$ and $H_2$ have coherently equivalent systems of maximal $k$-tori.
This anwers a question raised by A. Rapinchuk. Note such an example cannot occur over a number field \cite[th. 7.5]{PR}.

The examples use big fields ``\`a la Merkurjev'' constructed from the theory of Pfister forms, which is the main issue of the appendix.

\medskip
\noindent{\bf Acknowledgements.} The authors are grateful to the referee for his/her  careful reading
and also for valuable comments.

\bigskip

\section{Groups having coherently equivalent systems of maximal tori}\label{sect_coherent}

Our goal here is to reformulate by the notion of oriented type \cite{G1,R}
of embeddings of maximal tori Prasad-Rapinchuk's definition of groups
having  coherently equivalent systems of maximal $k$-tori.
First we generalize the notion of coherent embeddings from semisimple connected absolutely simple  to arbitrary reductive groups.

\subsection{Coherent embeddings}

\begin{definition}\label{defi_coherent} Let  $H_1, H_2$ be reductive $k$--groups which are isomorphic over $k_s$.
We fix a $k_s$--isomorphism $\varphi^\sharp: H_{1,k_s} \simlgr H_{2,k_s}$.

\smallskip

\noindent (a) Let $T_1$ be a maximal $k$--torus of $H_1$ and let $\iota_1: T_1 \to H_1$ be the natural inclusion.
A $k$-embedding $\iota:T_1 \to H_2$  is  called
{\it coherent} (relative to $\varphi^\sharp$) if there exists a  $k_s$-isomorphism
$\varphi: H_{1,k_s} \simlgr H_{2,k_s}$  of the form $\varphi  = \Int(h) \circ \varphi^\sharp$, with $h \in H_2(k_s)$
 such that $\iota = \varphi \circ \iota_1$.

\smallskip

\noindent (b) We say that $H_1$ and $H_2$ have coherently equivalent
systems of maximal $k$-tori (with respect to
$\varphi^\sharp$) if every maximal $k$-torus
$\iota_1: T_1 \to H_1$  admits a coherent $k$-embedding into $H_2$ (relative to $\varphi^\sharp$), and every maximal $k$-torus
$\iota_2: T_2 \to H_2$ admits a coherent (relative to $(\varphi^\sharp)^{-1}$) $k$-embedding into $H_1$.

\end{definition}

For reductive $k$-groups $H_1$ and $H_2$, let $\Isom_k(H_1,H_2)$ be the scheme of isomorphisms from $H_1$ to $H_2$ and
$\Isomext_k(H_1,H_2)$ be the quotient scheme of $\Isom_k(H_1,H_2)$ by the adjoint group $\ad(H_1)$.
Namely $$\Isomext_k(H_1,H_2)=\Isom_k(H_1,H_2)/\ad(H_1)=\ad(H_2)\setminus\Isom_k(H_1,H_2)$$ \cite[XXIV.2]{SGA3}.
We denote by $\pi$ the canonical map from $\Isom_k(H_1,H_2)$ to $\Isomext_k(H_1,H_2)$.
A $k$-point of $\Isomext_k(H_1,H_2)$ is called an \emph{orientation} between $H_1$ and $H_2$.

Let $H_1$, $H_2$ and $\varphi^\sharp$ be as in Definition \ref{defi_coherent}.
In the following lemma, we show that whether $H_1$ and $H_2$ have coherently equivalent systems of maximal $k$-tori relative to
$\varphi^\sharp$ actually only depends on $\pi(\varphi^\sharp)$.

\begin{lemma}\label{lem_coherent}
 We keep the setting of Definition \ref{defi_coherent}.

\smallskip

\noindent (1) Assume that $H_1$ and $H_2$ have coherently equivalent systems of maximal $k$-tori relative to
$\varphi^\sharp$. Then $H_2$
is an inner form of $H_1$ and $\pi(\varphi^\sharp)$ is a $k$-point of $\Isomext_k(H_1,H_2)$.

\smallskip

\noindent (2) Let $\psi^\sharp,\ \varphi^\sharp: H_{1,k_s} \simlgr H_{2,k_s}$ be $k_s$--isomorphisms such
that $\pi(\psi^\sharp)=\pi(\varphi^\sharp)$. Then the following are equivalent:

\smallskip

(i) The $k$--groups  $H_1$ and $H_2$ have coherently equivalent systems of maximal $k$-tori with respect to
$\varphi^\sharp$;

\smallskip

(ii) The $k$--groups $H_1$ and $H_2$ have coherently equivalent systems of maximal $k$-tori with respect to
$\psi^\sharp$.
\end{lemma}

\begin{proof}
$(1)$ Let $T_1$ be a maximal torus of $H_1$ and $\iota_1:T_1\to H_1$ be the natural inclusion and
$\iota:T_1\to H_2$ be a coherent embedding relative to $\varphi^\sharp$.
By definition, we have  some $h\in H_2(k_s)$ and $\varphi = \Int(h) \circ \varphi^\sharp$ such that $\varphi\circ \iota_1=\iota$.

Let $\Gamma=\Gal(k_s/k)$. For $t\in T_1(k_s)$ and $\sigma\in\Gamma$, we have $\iota(\leftidx{^\sigma}t)=\leftidx{^\sigma}\iota(t),$
which implies that $\varphi(\leftidx{^\sigma}\iota_1(t))=\leftidx{^\sigma}(\varphi(\iota_1(t))$.
Let $\leftidx{^\sigma}\varphi=\sigma\circ\varphi\circ\sigma^{-1}$.
The above equality implies that for all $\sigma\in\Gamma$, we have $\varphi^{-1}\circ\leftidx{^\sigma}\varphi$
is trivial on $\iota_1(T_1)$.
Since $\Aut_k(H_1,\Id_{T_1})=\ad(T_1)$ (\cite{SGA3} Exp. XXIV, Prop. 2.11), the group $H_1$ is an inner form of $H_2$
and $\pi(\varphi)$ is a $k$-point
of $\Isomext_k(H_1,H_2)$. As $\pi(\varphi^\sharp)=\pi(\varphi)$, we have $\pi(\varphi^\sharp)\in\Isomext_k(H_1,H_2)(k)$.

\smallskip

\noindent $(2)$ Let $h_2\in H_2(k_s)$ such that $\Int(h_2)\circ\psi^\sharp=\varphi^\sharp$.

\smallskip

\noindent $(i) \Longrightarrow (ii)$. Suppose that $H_1$ and $H_2$ have
coherently equivalent systems of maximal $k$-tori with respect to
$\varphi^\sharp$.
Then for every maximal torus $T_1$ of $H_1$, there is a coherent embedding $\iota:T_1\to H_2$ relative to $\varphi^\sharp$.
Let $h\in H_2(k_s)$ such that $\varphi\circ \iota_1=\iota$, where $\varphi=\Int(h)\circ\varphi^\sharp$.
Then we have  $\Int(hh_2)\circ\psi^\sharp\circ \iota_1=\iota$. Hence $\iota$ is also a coherent embedding relative to $\psi^\sharp$.
Therefore $H_1$ and $H_2$ have coherently equivalent systems of maximal $k$-tori with respect to
$\psi^\sharp$.
%A similar argument proves the other direction.

 \smallskip

\noindent $(ii) \Longrightarrow (i)$.  It is enough to interchange the roles of $\varphi^\sharp$ and $\psi^\sharp$.

\end{proof}

By Lemma \ref{lem_coherent}.(1), two reductive groups $H_1$ and $H_2$ have
coherently equivalent systems of maximal $k$-tori with respect to
$\varphi^\sharp$ only if $\pi(\varphi^\sharp)$ is a  $k$--point of
 $\Isomext(H_1,H_2)$. Therefore, we consider the coherently equivalent systems of maximal $k$-tori with respect to $\varphi^\sharp$
  only when  $\pi(\varphi^\sharp)$  is a $k$--point of $\Isomext(H_1,H_2)$.

Moreover, given an orientation $v^\sharp \in \Isomext(H_1,H_2)(k)$,
we say that  $H_1$ and $H_2$ have coherently equivalent systems of maximal $k$-tori relative to $v^\sharp$
if there is some $\varphi^\sharp : H_{1,k_s} \simlgr  H_{2,k_s}$ such that $\pi(\varphi^\sharp)=v^\sharp$ and
$H_1$ and $H_2$ have coherently equivalent systems of maximal $k$-tori
relatively to $\varphi^\sharp$. In fact by Lemma \ref{lem_coherent} (2), the groups $H_1$ and $H_2$
have coherently equivalent systems of maximal $k$-tori relative to $v^\sharp$ if and only if for all $\varphi^\sharp : H_{1,k_s} \simlgr  H_{2,k_s}$ with $\pi(\varphi^\sharp)=v$, the groups $H_1$ and $H_2$ have coherently equivalent systems of maximal $k$-tori
relatively to $\varphi^\sharp$.

\begin{fsremark}\label{remark_abusive} {\rm
 If  $\Autext(H_1)=1$ (or equivalently $\Autext(H_2)=1$),  we  have that  $\Isomext(H_1,H_2)=\Spec(k)$ so that
there is a canonical orientation $v_{can} \in \Isomext(H_1,H_2)(k)$.
By some abuse of notation, we can say that
$H_1$ and $H_2$ have coherently equivalent systems of maximal $k$-tori if they
have coherently equivalent systems of maximal $k$-tori with respect to $v_{can}$.
This concerns semisimple groups of type $A_1$, $B_n$, $C_n$, $E_7$, $E_8$, $F_4$ and $G_2$.
}
\end{fsremark}

An immediate corollary is the following:

\begin{corollary}
Suppose that $H_1$ and $H_2$ have coherently equivalent systems of maximal $k$-tori relative to
$\varphi^\sharp$. Then $H_1$ and $H_2$ share the same quasi-split form.
\end{corollary}

\begin{proof}
By Lemma \ref{lem_coherent}.(1), $H_1$ and $H_2$ are inner forms of each other, and hence share the same quasi-split form.
\end{proof}

\subsection{An equivalent definition}

Let $H$ be a reductive  $k$--group and   $H'$ be a quasi-split form
of $H$.  Let $(B',T')$ be a Killing couple of $H'$ and denote by $W'=N_{H'}(T')/T'$ the Weyl group of
$T'$. Let $\Psi'$ be the root datum $\Phi(H',T')$. For any maximal torus $T$ of $H$,
we let $\Phi(H,T)$ be the twisted root datum of $H$ with respect to $T$ \cite[Exp. XXII 1.9 and 1.10.]{SGA3}.
Denote by $W(T)=N_{H}(T)/T$ the Weyl group of $T$.
Let $\Isom_k(\Phi(H,T),\Psi')$ be the scheme of isomorphisms from $\Phi(H,T)$ to $\Psi'$ and let $\Isomext_k(\Phi(H,T),\Psi')$ be the quotient scheme
of $\Isom_k(\Phi(H,T),\Psi')$ by $W'$. Namely, we have
 \[
 \Isomext_k(\Phi(H,T),\Psi')=W'\backslash\Isom_k(\Phi(H,T),\Psi')=\Isom_k(\Phi(H,T),\Psi')/W(T).
 \]

There is a canonical isomorphism between $\Isomext_k(\Phi(H,T),\Psi')$ and $\Isomext_k(H,H')$ \cite[prop. 2.4]{Le}.
Therefore, given an orientation $v\in\Isomext_k(H,H')(k)$, we have a
corresponding $k$-point of $\Isomext_k(\Phi(H,T),\Psi')$, which is called an orientation between root data $\Phi(H,T)$ and $\Psi'$
and we still denote it by $v$.

Let $p$ be the canonical morphism from $\Isom_k(\Phi(H,T),\Psi')$ to $\Isomext_k(\Phi(H,T),\Psi')$.
Given an orientation $v\in\Isomext_k(\Phi(H,T),\Psi')(k)$, let $\Isomint_v(\Phi(H,T),\Psi')$ be the fiber of  $v$ under the morphism  $p$.
Note that $\Isomint_v(\Phi(H,T),\Psi')$ is a left $W'$-torsor.
Given an embedding $i: T \to H$, we defined its oriented type with respect
to $v$  as
$$\type_{v}(T,i)=\Isomint_v(\Phi(H,i(T)),\Psi') \in H^1(k,W')$$ (\cite[\S 2.2]{BGL}).

\begin{fsremark}{\rm If $H=H'$ and $v$ is induced by the identity map on $H$,
then for any maximal torus $T$ of $H'$ with the natural inclusion $i$, the map $T\to \type_{v}(T,i)$ is nothing but
 the composition of $\bigl(H'/N_{H'}(T')\bigr)(k) \xrightarrow{\varphi}  H^1\bigl(k,N_{H'}(T') \bigr) \to H^1(k,W')$ where
$\bigl(H'/N_{H'}(T')\bigr)(k)$ is seen as the set of maximal $k$--tori of $H'$ and  $\varphi$
denotes  the characteristic map. We recover then the definition of \cite{G1}.

Let us explain that compatibility. We are given a maximal $k$--torus  $T $ of $G'$, it is $k_s$--conjugated
to $T'$ so defines a $k$--point $x$ of $\bigl(H'/N_{H'}(T')\bigr)(k)$ \cite[\S 2.5]{G2} and furthermore
 $\varphi(x)$ is the class of the $N_{H'}(T')$-torsor $E=\mathrm{Transpst}_{H'}(T', T)$, that is the strict transporter of $T'$
 into $T$. The image of $\varphi(x)$ in $H^1(k,W')$ is then given by the class of the $W'$--torsor
 $F=E/T' = \mathrm{Transpst}_{H'}(T', T)/T'$. The point is that we have a natural $k$--morphism
 $\theta: E=\mathrm{Transpst}_{H'}(T', T) \to \mathrm{Isomint}_v( \Phi(G',T'), \Phi(G,T))$ which is $N'_{H'}(T')$--equivariant.
 Then it   induces a $W'$--map $\ol \theta : F \to \mathrm{Isomint}_v( \Phi(G',T'), \Phi(G,T))$. Since $F$ and
 $\mathrm{Isom}( \Phi(G',T'), \Phi(G,T))$ are $W'$--torsors, we conclude that $\ol \theta$ is an isomorphism of
 $W'$--torsors, whence the desired compatibility.
}
\end{fsremark}

We can interpret the coherently equivalent system of maximal tori in terms of orientation and types as follows.

\begin{definition}\label{def_same} Let  $H_1, H_2$ be reductive $k$--groups sharing the same quasi--split $k$--form $H'$
(equipped as before with $(B',T')$ and $W'$).
 Let $v\in\Isomext_k(H_1,H_2)(k)$, $v_2\in\Isomext_k(H_2,H')(k)$ and $v_1=v_2\circ v\in\Isomext_k(H_1,H')(k)$.
 We say that $H_1$, $H_2$ have the same oriented maximal $k$--tori with respect to
 $v$ if for each $\gamma \in H^1(k,W')$ the following are equivalent:

 \smallskip

 (i) There exists a $k$--embedding  $\iota_1: T \to H_1$ such that $\type_{v_1}(T,\iota_1) =\gamma$;

\smallskip

(ii) There exists a $k$--embedding $\iota_2: T \to H_1$ such that $\type_{v_2}(T,\iota_2) =\gamma$.

\end{definition}

\smallskip

\noindent
The above definition is independent of the choice of $v_2$.
To see this, we note that for a quasi-split group $H'$, there is  a section $s:\Autext(H')\to \Aut(H')$
whose image stabilizes the Killing couple $(B',T')$
(\cite{SGA3}, Exp. XXIV. 3.10). If we choose another orientation $u_2\in\Isomext(H_2,H')(k)$, then
there is $\alpha\in\Autext(H')(k)$ such that $\alpha\circ v_2=u_2$. Let $u_1=u_2\circ v$ and $a=s(\alpha)$.
Since $a$ stabilizes the Killing couple $(B',T')$, the automorphism $a$ induces an automorphism on $\Psi'$ and we denote it by $\ol{a}$.
The automorphism $\Int(\ol{a})$ of $W'$ induces an automorphism $\Int(\ol{a})^\ast:H^1(k,W')\to H^1(k,W').$
Let $\eta:\Isomint_{v_i}(\Phi(H_i,\iota_i(T)),\Psi')\to\Isomint_{u_i}(\Phi(H_i,\iota_i(T),\Psi')$ be defined
by $\eta(f)=\ol{a}\circ f$. Then $\Int(\ol{a})^\ast(\type_{v_i}(T,\iota_i))$ is nothing but  $\eta(\type_{v_i}(T,\iota_i))$.
Therefore, the definition does not depend on the choice of $v_2$.
Since a different choice of $H'$ will induce an one-to-one correspondence on types,
the definition does not depend on the choice of $H'$ either.

\smallskip

\begin{sproposition}\label{prop_orient}
Let $H_1, H_2,H'$ be reductive $k$--groups as in Definition \ref{def_same}.
Let $v\in\Isomext_k(H_1,H_2)(k)$.
 Then $H_1$ and $H_2$
have coherently equivalent systems of maximal $k$-tori relative to $v$
 if and only if $H_1,H_2$ have
the same oriented maximal $k$--tori with respect to the orientation $v$.
\end{sproposition}

We first prove the following lemma.

\begin{lemma}\label{lem_orient}
Let $H_1, H_2,H'$ be reductive $k$--groups as in Definition \ref{def_same}.
Let $v\in\Isomext_k(H_1,H_2)(k)$.
Then the following are equivalent:

\smallskip

\noindent (1) The $k$--groups  $H_1,H_2$ have
the same oriented maximal $k$--tori with respect to the orientation $v$;

\smallskip

\noindent (2) For every $k$--torus $T$ of rank $\mathrm{rank}(H')$ and every embedding $\iota_1:T \to H_1$,
there exist an embedding $\iota_2:T\to H_2$ and a $k$-isomorphism
$\theta:\Phi(H_1,\iota_1(T))\to \Phi(H_2,\iota_2(T))$ with orientation $v$; and
for every $k$--torus $T$ of rank $\mathrm{rank}(H')$ and every embedding $\iota_2:T\to H_2$, there exist an
embedding $\iota_1:T\to H_1$ and a $k$-isomorphism
$\theta:\Phi(H_1,\iota_1(T))\to \Phi(H_2,\iota_2(T))$ with orientation $v$.
\end{lemma}

\begin{proof}
Let $v_1$ and $v_2$ be as in Definition \ref{def_same}.

\smallskip

\noindent $(1) \Longrightarrow (2)$.
We suppose that $H_1,H_2$ have
the same oriented maximal $k$--tori with respect to the orientation $v$.
Given $(T,\iota_1)$, there is $(T,\iota_2)$ such that $\type_{v_1}(T,\iota_1)=\type_{v_2}(T,\iota_2) \in H^1(k,W')$.
In other words, there is an isomorphism of $W'$-torsors
$$\eta:\Isomint_{v_2}(\Phi(H_2,\iota_2(T)),\Psi')\simlgr \Isomint_{v_1}(\Phi(H_1,\iota_1(T)),\Psi').$$
Let $f\in\Isomint_{v_2}(\Phi(H_2,\iota_2(T)),\Psi')(k_s)$ and let $\theta=f^{-1}\circ \eta(f)$.
Since $\eta$ is a $k$-isomorphism of $W'$-torsors, we have $\eta(f')=f'\circ\theta$ for all $f'\in\Isomint_{v_2}(\Phi(H_2,\iota_2(T)),\Psi')(k_s)$
and $\theta:\Phi(H_1,\iota_1(t))\to \Phi(H_2,\iota_2(t))$ defined over $k$. From our construction, it is clear that $\theta$ is with orientation $v$.

% The other direction is easy, and we left it to readers.

\smallskip

\noindent $(2) \Longrightarrow (1)$. Let $\gamma\in H^1(k,W')$. Assume that there exists a $k$--embedding $\iota_1:T \to H_1$
such that $\type_{v_1}(T,\iota_1)=\gamma$.
 Our assumption provides a $k$-embedding $\iota_2:T\to H_2$ and a $k$-isomorphism
$\theta:\Phi(H_1,\iota_1(T))\to \Phi(H_2,\iota_2(T))$ with orientation $v$.
Since $v_1=v_2\circ v$,
it induces an isomorphism of $W'$--torsors
$$
\Isomint_{v_2}(\Phi(H_2,\iota_2(T)),\Psi') \simlgr \Isomint_{v_1}(\Phi(H_1,\iota_1(T)),\Psi')
$$
so that $\type_{v_2}(T,\iota_2)=\gamma$.
By interchanging the roles of $H_1$ and $H_2$, we get (1).
\end{proof}

\smallskip

We can proceed to the proof of Proposition \ref{prop_orient}.

\begin{proof}
Suppose that $H_1,H_2$ have the same oriented maximal $k$--tori with respect to the orientation $v$.
Fix a group isomorphism $\varphi^\sharp:H_{1,k_s}\xrightarrow{\sim} H_{2,k_s}$ such that $\pi(\varphi^\sharp)=v$.
Let $T$ be a maximal torus of $H_1$ and $\iota_1$ is the natural inclusion.
By Lemma \ref{lem_orient}, there is an embedding $\iota_2:T\to H_2$ and a $k$-isomorphism $\theta:\Phi(H_1,\iota_1(T))\to \Phi(H_2,\iota_2(T))$ with orientation $v$.
Let $\Isom(H_1,T;H_2,\iota_2(T))$ be the scheme of isomorphisms from $H_1$ to $H_2$ which send $T$ to $\iota_2(T)$. (For the notation, see \cite{SGA3} Exp. XXIV, \S 2.)
Let $\varphi\in\Isom(H_1,T;H_2,\iota_2(T))(k_s)$ be a lifting of $\theta$, i.e. $\varphi|_{T}=\theta$.
Let $\iota=\varphi\circ \iota_1$.
As $\theta$ is defined over $k$, the isomorphism $\varphi|_{T}$ is an $k$-isomorphism between $T$ and $\iota_2(T)$.
Hence $\iota:T\to H_2$ is a $k$-embedding.
From our construction, we have $\pi(\varphi)=v=\pi(\varphi^\sharp)$.
Therefore there is $h_2\in H_2(k_s)$ such that $\varphi=\Int(h_2)\circ\varphi^\sharp$ and the embedding $\iota$ is coherent to $\varphi^\sharp$.

On the other hand, given a maximal $k$--torus $T$ of $H_2$ and $\iota_2$ be the natural inclusion, the same argument as above also works.
Therefore $H_1$ and $H_2$ have coherently equivalent systems of maximal $k$-tori relative to $v$.

Suppose that $H_1$ and $H_2$ have coherently equivalent systems of maximal $k$-tori relative to $v$.
Let $T$ be a torus and $\iota_1:T\to H_1$ be an embedding. Since the type only depends on the image of $\iota_1$,
we can identify $T$ with $\iota_1(T)$ and let $\iota_1$ be the natural inclusion.
Let $\iota:T\to H_2$ be a coherent embedding and $\varphi=\Int(h_2)\circ\varphi^\sharp$ such that $\varphi\circ \iota_1=\iota$.
Since $\varphi|_{T}=\iota$, the isomorphism $\varphi$ induces a $k$-isomorphism $\theta:\Phi(H_1,T)\to\Phi(H_2,\iota(T))$.
Since $\pi(\varphi)=\pi(\varphi^\sharp)=v$, the map $\theta$ is with orientation $v$.

Given a $k$--torus $T$ and an embedding $\iota_2:T\to H_2$, the same argument also works for $H_2$. By Lemma \ref{lem_orient},
the groups $H_1$ and $H_2$ have the same oriented maximal $k$-tori with respect to the orientation $v$.
\end{proof}

\section{Applications to groups of type $G_2$}

We start with the $A_1$-case.

\begin{slemma}
\label{lem_coh_A1} Let $Q_1$ and $Q_2$ be two quaternion algebras over $k$ and put $H_i=\SL_1(Q_i)$ for $i=1,2$.
Then the following are equivalent:

\smallskip

\noindent (1)  For each quadratic \'etale algebra $k'$, $Q_1 \otimes_k k'$ splits if and only if $Q_2 \otimes_k k'$ splits.

\smallskip

\noindent (2) The $k$--groups $H_1$ and $H_2$ have coherently equivalent systems
of maximal $k$-tori (in the sense of Remark \ref{remark_abusive}).

\smallskip

\noindent (3) The $k$--groups $H_1$ and $H_2$ have the same maximal tori.

\end{slemma}

\begin{proof} By Proposition \ref{prop_orient}, (2) is the same than $H_1$ and $H_2$
have same oriented maximal $k$--tori.

\smallskip

\noindent $(1) \Longrightarrow (2)$. The Weyl group of $\SL_2$ is $\ZZ/2\ZZ$.
Let  $\gamma$ be a class in $H^1(k,\ZZ/2\ZZ)$, that is  the isomorphism class of  a quadratic \'etale $k$--algebra $k'$.
We assume that there exists a maximal $k$--torus embedding $\iota_1: T \to H_1=\SL_1(Q_1)$ of type $[k']$.
Then $T=R^1_{k'/k}(\GG_m)$ and $\iota_1$ arises from a $k$--algebra map $k' \to Q_1$.
Then $k'$ splits $Q_1$, it splits $Q_2$ as well according to our assumption. Hence there exists a $k$--algebra map $k' \to Q_2$.
This gives rise to a  maximal $k$--torus embedding $\iota_2: T \to H_2=\SL_1(Q_2)$ of type $[k']$.
This shows that $H_1$ and $H_2$ have same oriented maximal $k$--tori.

\smallskip

\noindent $(2) \Longrightarrow (3)$. Obvious.

\smallskip

\noindent $(3) \Longrightarrow (1)$. Let $k'$ be a quadratic \'etale $k$--algebra and assume that $k'$ splits $Q_1$.
We put $T=R^1_{k'/k}(\GG_m)$ and we have seen that it implies that there is a
maximal $k$--torus embedding $\iota_1: T \to H_1=\SL_1(Q_1)$
Our assumption implies that there is a
maximal $k$--torus embedding $\iota_2: T \to H_2=\SL_1(Q_2)$. Since $k'$ splits $T$, it follows that $k'$ splits $H_2$ and
then splits $Q_2$.
We have shown that $Q_1 \otimes_k k'$ splits if and only if $Q_2 \otimes_k k'$ splits.
\end{proof}

Together with Proposition \ref{prop_pure}, Lemma \ref{lem_coh_A1} provides examples of non-isomorphic
semisimple simply connected groups of type $A_1$ having coherently equivalent systems
of maximal $k$-tori.
We come now to the octonionic case.

\begin{sproposition}\label{prop_cd3}
There exists a field $F$ with two octonions $F$--algebras $C_1$, $C_2$
satisfying the following conditions:

\smallskip

(1) $F$ is $2$--special and $\mathrm{cd}(F)=3$;

\smallskip

(2) $C_1$ and $C_2$ are non-isomorphic and both non-split.

\smallskip

(3) The $F$--groups $H_1=\Aut(C_1)$ and $H_2=\Aut(C_2)$
have coherently equivalent systems of maximal $F$-tori (in the sense of Remark \ref{remark_abusive}).

\end{sproposition}

\begin{proof} 
Proposition \ref{prop_skip} of the appendix applied to $n=3$
provides a field $F$ of odd characteristic which is $2$--special (that is its absolute Galois group is a pro-$2$--group)
and of cohomological dimension $3$ together with two non-isometric anisotropic $3$-Pfister forms $\varphi_1$, $\varphi_2$ over $F$
 such that for each $\delta \in F^\times \setminus F^{\times 2}$,
$\varphi_1$ and $\varphi_2$ are split by $F(\sqrt{\delta})$.

We denote by $C_i$ the unique octonion $F$--algebra whose norm form is $\varphi_i$ for
$i=1,2$. Then $C_1$ and $C_2$ are non--isomorphic and both non-split.

It remains to establish property (3). Again (3) is equivalent to that $H_1$ and $H_2$
have the same oriented maximal $k$--tori. Let $H_0$ be the split $F$--group of type $G_2$ and let $T_0$ be a
maximal $F$-split torus of $G_0$. We have $W_0=N_{H_0}(T_0)/T_0=   S_2 \times S_3$ so that
$H^1(F, W_0)$ classifies isomorphism classes of couples $(F',L)$ where $F'$ (resp. $L$)
is an \'etale quadratic (resp. cubic) $F$--algebra.
We are given a class $\gamma \in H^1(F, S_2 \times S_3)$,
namely a class $[(F', L)]$. Since $F$ is $2$--special, we have that $L=F \times E$ where $E$ is a quadratic  \'etale $F$--algebra.
We assume that $H_1=\Aut(C_1)$ admits an $F$--embedding $\iota_1: T \to H_1$ as maximal torus of type
$\gamma$. By \cite[lemma 4.2.1]{BGL}, we have that $T= \bigl( R^1_{F''/F}(\GG_m) \times_F R^1_{F'/F}(\GG_m) \bigr)/ \mu_2$
where $F''$ is the   quadratic  \'etale $F$--algebra  defined by  $[F']+[F'']=[E] \in H^1(F, \ZZ/2\ZZ)$.
By our embedding criterion ({\it ibid}, prop. 4.4.1), we have that $C_1$ is split by $F''$ and $F'$,
or equivalently  $\varphi_1$ is split by $F''$ and $F'$.
Hence $\varphi_2$ and $C_2$ are split by $F''$ and $F'$; the same criterion yields
that there exists a $F$--embedding $\iota_2 : T \to H_2=\Aut(C_2)$ of type $\gamma$.
\end{proof}

\bigskip

A more elaborated example is the following one.

\begin{stheorem} \label{main}
Let $k$ be a field containing a primitive $12$th--root of unity.
We assume that  $k$ is a $2$-special field of cohomological dimension $2$ and
that there exist two non-isomorphic quaternion division algebras $Q_1$, $Q_2$
such that $Q_1$  (resp. $Q_2$) contains all quadratic field extensions of $k$.
We put $K=k((t))$ and consider the octonion $K$--algebras $C_i=C(Q_{i,K},t)$ for
$i=1,2$ defined  by the Cayley-Dickson doubling process, that is \cite[\S 2.1]{SV}
$$
C_i= Q_{i,K} \oplus Q_{i,K} u, \quad u^2=t.
$$

\noindent (1) The octonion $K$--algebras $C_1$ and $C_2$ are non-isomorphic and both non-split.
Furthermore for each $\delta \in K^\times \setminus K^{\times 2}$,  the octonion $K$-algebras
$C_1$ and $C_2$ are split by $K(\sqrt{\delta})$.

\smallskip

\noindent (2) The $K$--groups $H_1=\Aut(C_1)$ and $H_2=\Aut(C_2)$
have coherently equivalent systems of maximal $K$-tori.
\end{stheorem}

\begin{sremark}{\rm Note that the input of Theorem \ref{main} can be  provided by Proposition \ref{prop_skip}
of the appendix for $n=2$.
}
\end{sremark}

\begin{proof}  We remind first that $K$ is of cohomological dimension $3$ \cite[II, \S 4.3, prop. 12]{Se}.
Also the norm form of $C_i$ is $N_i:=\langle 1,-t \rangle \otimes n_{Q_i,K}$ for $i=1,2$.
Since the quaternionic norm $n_{Q_1}$, $n_{Q_2}$ are anisotropic and
non--isometric, Springer's criterion shows that the $K$--forms $N_1$ and $N_2$ are anisotropic and
non--isometric \cite[VI.1.9]{La}.

Let $K'=K(\sqrt{\delta})$ be a  quadratic field extension of $K$.
If $K'$ is unramified, we may assume that $\delta \in k$, so that $K'=k(\sqrt{\delta})((t))= k'((t))$.
Since $k'$ occurs as subfield of $Q_1$, it follows that $K'$ occurs as composition subalgebra of
$C_1$, so splits $C_1$.
If $K'$ is ramified, we may assume that $\delta = t \delta_0$ with $\delta_0 \in k$.
We have $N_{1,K'}= \langle 1,-(\sqrt{\delta})^2 \delta_0 \rangle \otimes n_{Q_i,K}
= \Bigl( \langle 1, - \delta_0 \rangle \otimes n_{Q_1} \Bigr)_{K'}$.
If $\delta_0 \in k^{\times 2}$, $N_{1,K'}$ splits.
If  $\delta_0 \not \in k^{\times 2}$, we can write $Q_1=(\delta_0, \delta_1)$
for some $\delta_1 \in k^\times$.
It follows that $N_{1,K'}
=  \langle \langle  \delta_0, \delta_0, \delta_1 \rangle \rangle_{K'}$
 splits since $-1$ is a square in $k$.
We conclude from this case by case discussion that $N_{1}$ and $C_1$ are split
over $K'$ (and similarly for $N_2$ and $C_2$).

\smallskip

It remains to check property (2).
We are given a class $\gamma \in H^1(K, S_2 \times S_3)$,
namely a class $[(K', L)]$ where $K'$ (resp. $L$) is a quadratic (resp. cubic) \'etale $K$--algebra
such that $H_1=\Aut(C_1)$ admits a maximal torus embedding $\iota_1:T \to \Aut(C_1)$
of type $\gamma$. Since $C_1$ is not split,  $K'$ is a field \cite[prop. 4.3.1.(1)]{BGL} and we
write it as $K'=K(\sqrt{d})$.
If $L=K \times E$ with $E$ an \'etale quadratic algebra,
(1) and  the same argument as in the proof of Proposition \ref{prop_cd3} shows that
 there is a $K$--embedding $\iota_2: T \to H_2$ of type $\gamma$.
We can focus then on the case when $L$ is a cubic field extension of
$K$. Since $k$ is a $2$--special field, $L$ is ramified so that
$L= K( \sqrt[3]{a_0 t})$ for some $a_0 \in k^\times$.
But $\sqrt[3]{a_0} \in K$ so $L= K( \sqrt[3]{t})$ and is a Galois cubic extension since
$K$ contains a primitive $3$--root of unity.
We apply now our embedding criterion \cite[prop. 5.2.6]{BGL} based on
the work of Haile-Knus-Rost-Tignol \cite{HKRT}.
There exists a 3-dimensional $K'/K$-hermitian form $h=\langle -b,-c, bc \rangle$ of trivial hermitian discriminant
such that $C_1 \cong  C(K',{K'}^3,h)$ arises by  Jacobson's construction
and an element $\lambda \in L^\times$ such that $N_{L/K}(\lambda) \in K^{\times 2}$ and such that
the quadratic $K$--form  $\langle \langle d \rangle  \rangle \otimes t_{L/K}\bigl(\langle \lambda \rangle \bigr)$ is
isometric to $\langle \langle d \rangle \rangle \otimes \langle -b,-c, bc \rangle$.
The notation  $t_{L/K}\bigl(\langle \lambda \rangle \bigr)$ means
the form $L \cong K^3 \to K$, $x \to \mathrm{Tr}_{L/K}(\lambda x^2)$.

\begin{claim}
 $\lambda$ is a square in $L$.
\end{claim}

Up to a square of $L^\times$, we can write  $\lambda= \lambda_0 \,  t^{\frac{r}{3}}$ with $\lambda_0 \in k^\times$ and $r=0$ or $1$.
We have $N_{L/K}(\lambda) =  \lambda_0^3 \,  t^{\frac{3r}{3}} \in K^{\times 2}$.
By taking the valuation, $r$ is even so is zero. Since the map $k^\times/ k^{ \times 2} \to K^\times/ K^{ \times 2}$ is injective,
it follows that  $\lambda_0 \in k^{ \times 2}$, hence  $\lambda$ is a square in $L$.
The Claim is proven so we may assume that $\lambda=1$.
By writing $L=K \oplus K t^{\frac{1}{3}} \oplus K t^{\frac{2}{3}}$, we get that
the matrix of $t_{L/K}\bigl(\langle 1 \rangle \bigr)$ is
$$
\left[\begin{array}{ccc}
3&0& 0 \\
0&0&3t \\
0&3t&0 \\
\end{array}
\right].
$$
Since $3$ is a  square in $k$,
we have that $t_{L/K}\bigl(\langle 1 \rangle \bigr) \cong  \langle 1,1,-1\rangle$
so that
$$
\langle \langle d \rangle \rangle \otimes
 \langle 1,1,-1\rangle \cong
\langle \langle d \rangle \rangle \otimes \langle -b,-c, bc \rangle.
$$
The Pfister form $\langle \langle d \rangle \rangle \otimes \langle 1, -b,-c, bc \rangle$ is isotropic since
it contains as subform the isotropic form $\langle \langle d \rangle \rangle \otimes \langle -b,-c, bc \rangle$.
Hence  $\langle \langle d \rangle \rangle \otimes \langle 1, -b,-c, bc \rangle$ is hyperbolic.
On the other hand, we have the orthogonal decomposition
$$
C_1 \cong C(K',K^3,h)= K' \oplus (K')^3
$$
and the quadratic form associated to the hermitian form $h$
is the restriction of $N_{1}$ to  $(K')^3$.
It follows that
$$
N_{1} \cong \langle \langle  d \rangle \rangle  \otimes \langle 1,- b, -c, bc \rangle
$$
hence is $N_1$ hyperbolic. It  is then a contradiction and we conclude that  the cubic field
case extension does not occur.
\end{proof}

\begin{sremarks}{\rm
(a) The $k((t))$--groups $H_1$ and $H_2$ are defined over $k(t)$.
One natural question is whether the relevant $k(t)$--groups have
coherently equivalent systems of maximal $k(t)$-tori.

\smallskip

\noindent (b) By using Meyer's refinement of Garibaldi-Saltman's construction \cite{Mr},
one can construct a field $k$ which is $2$--special and of cohomological dimension $2$ and which has
 infinitely many quaternion division algebras $(Q_i)_{i \in I}$ (pairwise non-isomorphic)
such that  $Q_i \otimes_k k(\sqrt{\delta})$ is split for each $i \in I$ and
each $\delta \in k^\times \setminus k^{\times 2}$. Theorem \ref{main} provides then infinitely
many (pairwise non-isomorphic)  $k((t))$--groups of type $G_2$ having pairwise coherently equivalent systems of maximal $k((t))$-tori.
}
\end{sremarks}

\medskip

\section{Appendix: Pfister forms}

Let $k$ be a field of odd characteristic. We start with a variation  on the pure
subform theorem on Pfister forms. Given $a_1,...,a_n \in k^\times$,
we denote by  $\langle \langle a_1, \dots, a_n \rangle \rangle= \langle 1, - a_1 \rangle \otimes \dots \otimes \langle 1, - a_n \rangle$
the $n$--fold Pfister form; we warn the reader that the notation in Lam's book is with $\langle 1, a_1 \rangle \otimes \dots \otimes \langle 1, a_n \rangle$.

\begin{sproposition} \label{prop_pure}
Let $n \geq 2$ be an integer. Let $\varphi$ be a $n$-Pfister form and denote by
$\varphi'$ its pure subform.
Let $\delta \in k^\times \setminus k^{\times 2}$.
Then the following are equivalent:

\smallskip

(i) $\varphi_{k(\sqrt{\delta})}$ is hyperbolic;

\smallskip

(ii) The form $\varphi' \perp \langle \delta \rangle$ is isotropic.

\smallskip

(iii) There exists $b_2,\dots, b_n \in k^\times$ such that $\varphi \cong \langle \langle \delta, b_2,\dots , b_n\rangle \rangle$.

\end{sproposition}

Note that (ii) is equivalent to the fact that $-\delta$ is represented by $\varphi'$.

\begin{proof}
If $\varphi$ is hyperbolic, all assertions hold so that we can assume than $\varphi$
is anisotropic.

\smallskip

\noindent  $(i) \Longrightarrow (ii)$. Then we can write $\varphi= \langle a, - \delta a \rangle \perp \psi$ \cite[VII.3.1]{La}.
Since $\varphi$ is multiplicative, it follows that $\varphi= a \varphi= \langle 1, - \delta  \rangle \perp a\psi$.
By Witt cancellation, it follows that  $\varphi'=  \langle - \delta  \rangle \perp a\psi$.
Hence  $\varphi' \perp \langle \delta \rangle =  \langle \delta, - \delta  \rangle \perp a\psi$ is isotropic.

\smallskip

\noindent  $(ii) \Longrightarrow (iii)$. It  is  the pure subform theorem \cite[th. X.1.5]{La}.

\smallskip

\noindent  $(iii) \Longrightarrow (i)$. Obvious.
\end{proof}

We do now a variation on Garibaldi-Saltman's construction \cite[Example 2.1]{GaS}.

\begin{sproposition} \label{prop_skip} Let $n \geq 2$ be an integer.
Let $\varphi_1, \varphi_2$ be both  anisotropic $n$--Pfister forms.
We assume that $\psi=(\varphi_1 \perp - \varphi_2)_{an}$ is of dimension $2^n$.
Then there exists a field extension $F/k$ satisfying the following properties:

\smallskip

(i) $\varphi_{1,F}, \varphi_{2,F}, \psi_F$ are  anisotropic;

\smallskip

(ii) For each $\delta \in F^\times\setminus F^{\times 2}$,  $\varphi_{1,F(\sqrt{\delta})}$,
$\varphi_{2,F\sqrt{\delta})}$, $\psi_{F(\sqrt{\delta})}$ are split;

\smallskip

(iii) $F$ is $2$--special, i.e. its absolute Galois group is a pro-$2$--group;

\smallskip

(iv) $\mathrm{cd}(F)=n$.

\end{sproposition}

According to Arason-Pfister \cite[Kor. 3]{AP} (or \cite[X.4.34]{La}), we know that the form $\psi$
is similar to a $n$--Pfister form.
Garibaldi-Saltman's original construction is the case $n=2$ without the refinements (iii) and (iv). By the dictionnary between quaternion algebras
and $2$--Pfister forms, it permits to construct
non-isomorphic quaternion algebras $Q_1$, $Q_2$ over a field $F$ which are split by $F(\sqrt{\delta})$
for each $\delta \in F^\times \setminus F^{\times 2}$.

\begin{slemma} \label{lem_skip} Under the assumptions of Proposition \ref{prop_skip}, let $\delta \in k^\times \setminus {k^\times}^2$.
Denote by $E^{\delta}_i$ the function field of the projective quadric $\{ \varphi'_i \perp \langle  \delta \rangle =0 \}$ for $i=1,2$
and put $E^{\delta}=E_1^{\delta}.E_2^{\delta}$.
Then $\varphi_{1,E}$, $\varphi_{2,E}$, $\psi_E$ are  anisotropic and are split by $E(\sqrt{\delta})$.
\end{slemma}

\begin{proof}
The point is that  $\varphi'_i \perp \langle \delta \rangle$ is of discriminant $\delta$
so is not  similar to a Pfister form for $i=1,2$.
By \cite[Cor. X.4.10.(3)]{La}, it follows that $\varphi_1$ (resp. $\varphi_2$, $\psi$) remains anisotropic
over $E_1^{\delta}$ and $E^{\delta}=E_1^{\delta}.E_2^{\delta}$.
Also Proposition \ref{prop_pure}, $(ii)\Longrightarrow (i)$, ensures that
 $\varphi_{1,E}, \varphi_{2,E}$ are split by $E(\sqrt{\delta})$
and so is  $\psi_E$.
\end{proof}

\medskip

We proceed now to the proof of Proposition \ref{prop_skip}.

\begin{proof} We shall   construct a tower of fields  $k_0 =k \subset k_1 \subset k_2 \subset \dots$
\`a la Merkurjev \cite{Me}.
We denote by $k_1$ the composition of the function fields $E^{\delta}$, defined in Lemma \ref{lem_skip}, for $\delta$ running over $k^\times \setminus k^{\times 2}$.
In other words, $k_1$ is the inductive limit of the fields
$E^{\lambda}=E^{\delta_1}.\dots E^{\delta_n}$ where $\lambda=(\delta_1, \dots, \delta_n)$
runs over the finite subsets of $k^\times \setminus k^{\times 2}$.

\begin{sclaim}\label{claim1} The quadratic forms $\varphi_{1,k_1}, \varphi_{2,k_1}, \psi_{k_1}$ are  anisotropic
and are split by $k_1(\sqrt{\delta})$ for
each $\delta \in k^\times \setminus (k^\times)^2$.
\end{sclaim}

\noindent By construction, $k$ is algebraically closed in $k_1$, so that
$k^\times / k^{\times 2}$ injects in  $k_1^\times / k_1^{\times 2}$
Lemma \ref{lem_skip} shows that quadratic forms $\varphi_{1,E^\lambda}$, $\varphi_{2,E^\lambda}$, $\psi_{E^\lambda}$
are  anisotropic for each finite tuple $\lambda$ of elements of $k^\times \setminus k^{\times 2}$.
It follows that the quadratic forms $\varphi_{1,E}$, $\varphi_{2,E}$, $\psi_E$ are  anisotropic.
The fact that $\varphi_{1,E}$, $\varphi_{2,E}$, $\psi_E$ are  anisotropic and are split by $E(\sqrt{\delta})$
follows from the construction.
The Claim is proven.

The two other steps of the construction are standard. We denote by $k_2$ the composition of the function fields
of the projective quadrics $\{q=0\}$ where $q$ runs over the $(n+1)$-Pfister $k_1$--forms.
Now we take $k_3$ as a maximal separable algebraic odd extension of $k_2$.

\begin{sclaim}\label{claim2} The quadratic forms $\varphi_{1,k_3}, \varphi_{2,k_3}, \psi_{k_3}$ are  anisotropic.
\end{sclaim}

\noindent The passage from $k_2$ to $k_3$ works by Springer's odd extension theorem.
For $k_2$, as before, it is enough to justify  that the anisotropy is preserved
on the function field $k_1(q)$ of a projective quadric arising from a $(n+1)$--Pfister $k_1$--form $q$.
Since our forms are similar to $n$-Pfister forms, this works
 granting \cite[Cor. X.4.13]{La}. The Claim is proven.

We put then $F_0=k$, $F_1=k_3$, construct $F_2$ from $F_1$ as $k_3$ from $k=F_0$ and so on.
We put $F= \limind_{n \geq 0} F_n$ and shall check the requested properties.

Claim \ref{claim2} insures that  $\varphi_{1,F}, \varphi_{2,F}, \psi_{F}$ are  anisotropic.
Claim \ref{claim1} guarantees that for each $\delta \in F^\times\setminus F^{\times 2}$,  $\varphi_{1,F(\sqrt{\delta})}$,
$\varphi_{2,F(\sqrt{\delta})}$, $\psi_{F(\sqrt{\delta})}$ are split.
By construction, $F$ has no non-trivial separable odd finite field extension, hence
$F$ is $2$--special.
For determining the cohomological dimension of $F$, we use the quadratic part of the Milnor conjecture, namely
the isomorphisms $I^r(F)/ I^{r+1}(F) \simlgr H^{r}(F, \ZZ/2 \ZZ)$ established by Orlov-Vishik-Voevodsky \cite[th. 4.1]{OVV}
(see also \cite{Mo}).

By construction, each $(n+1)$-Pfister over $F$ is hyperbolic, so that  $I^{n+1}(F)=0$.
Since $\varphi_{1,F}$ is an anisotropic $n$-Pfister form,
it defines a non-trivial class in $I^n(F)/ I^{n+1}(F)$
so that $H^{n}(F, \ZZ/2 \ZZ) \not =0$. This implies that $\mathrm{cd}(F) \geq n$.
In the other hand,  we have that $H^{n+1}(F, \ZZ/2 \ZZ)=0$.
Since $F$ is $2$--special, this implies that $\mathrm{cd}(F) \leq n$ \cite[I, \S 4, prop. 21]{Se}.
Thus  $\mathrm{cd}(F) = n$.
\end{proof}

\begin{sremark}
 {\rm  For each $n\geq 2$, there are examples of fields $k$ satisfying the assumptions of Proposition
\ref{prop_skip}.
Let $k_0$ be a field having an anisotropic $(n-1)$--Pfister form $\varphi_0$ (e.g $k_0=\mathbb{R}$).
Put $k= k_0(t_1,t_2)$, $\varphi_1= \langle 1, -t_1 \rangle \otimes \varphi_{0,k}$,
 $\varphi_2=  \langle 1, -t_2 \rangle \otimes \varphi_{0,k}$. Then $\varphi_1$ and
$\varphi_2$ are anisotropic (e.g. by  Springer criterion over $k((t_1))((t_2))$, see \cite[prop. VI.1.9]{La}).
We have $\varphi_1 \perp - \varphi_2 =  \langle 1, -1 \rangle \otimes \varphi_{0,k} \perp
\langle t_1, -t_2 \rangle \otimes \varphi_{0,k}$.
Since $\langle t_1, -t_2 \rangle \otimes \varphi_{0,k}$ is anisotropic, we have that
$(\varphi_1 \perp - \varphi_2)_{an}=\langle t_1, -t_2 \rangle \otimes \varphi_{0,k}$
is of dimension $2^n$.
}
\end{sremark}

\bigskip

\medskip

\end{document}